\documentclass{article}
\usepackage[utf8]{inputenc}
\usepackage{amsmath,amssymb,amsthm, amsfonts}
\usepackage{etaremune, enumerate, float, verbatim}
\usepackage{algorithm}
\usepackage{algorithmic}
\usepackage{hyperref}
\usepackage{graphicx}
\usepackage{url}
\usepackage[mathlines]{lineno}
\usepackage{dsfont} 
\usepackage{tikz, graphicx, subcaption, caption}
\usepackage[algo2e,ruled,vlined]{algorithm2e}
\usepackage{mathrsfs,amsmath}
\usepackage{color}
\definecolor{red}{rgb}{1,0,0}

\definecolor{blue}{rgb}{0,0,.7}

\definecolor{green}{rgb}{0,.6,0}

\definecolor{purp}{rgb}{.5,0,.5}

\numberwithin{figure}{section}   

\newtheorem{thm}{Theorem}[section]
\newtheorem{cor}[thm]{Corollary}
\newtheorem{lem}[thm]{Lemma}
\newtheorem{prop}[thm]{Proposition}
\newtheorem{conj}[thm]{Conjecture}
\newtheorem{obs}[thm]{Observation}

\newtheorem{quest}[thm]{Question}

\theoremstyle{definition}
\newtheorem{rem}[thm]{Remark}

\theoremstyle{definition}
\newtheorem{defn}[thm]{Definition}

\theoremstyle{definition}
\newtheorem{ex}[thm]{Example}

\newcommand{\ol}{\overline}

\newcommand{\bit}{\begin{itemize}}
\newcommand{\eit}{\end{itemize}}
\newcommand{\ben}{\begin{enumerate}}
\newcommand{\een}{\end{enumerate}}
\newcommand{\beq}{\begin{equation}}
\newcommand{\eeq}{\end{equation}}
\newcommand{\bea}{\begin{eqnarray*}} 
\newcommand{\eea}{\end{eqnarray*}}
\newcommand{\bpf}{\begin{proof}}
\newcommand{\epf}{\end{proof}\ms}
\newcommand{\bmt}{\begin{bmatrix}}
\newcommand{\emt}{\end{bmatrix}}
\newcommand{\ms}{\medskip}

\newcommand{\cp}{\, \Box\,}

\newcommand{\noi}{\noindent}

\newcommand{\dist}{\operatorname{dist}}
\newcommand{\diam}{\operatorname{diam}}

\newcommand{\Z}{\operatorname{Z}}
\newcommand{\Zp}{\operatorname{Z}_+}

\newcommand{\mcs}{\mathcal{S}}
\newcommand{\mct}{\mathcal{T}}

\newcommand{\pd}{\gamma_P}

\newcommand{\gbar}{\ol\gamma}
\newcommand{\pdbar}{\ol\pd}
\newcommand{\xbar}{\ol X}
\newcommand{\te}{\mathscr{PD}^{TJ}}
\newcommand{\tar}{\mathscr{PD}^{TAR}}
\newcommand{\xte}{\mathscr{X}^{TJ}}
\newcommand{\xtar}{\mathscr{X}^{TAR}}

\newcommand{\xxo}{x_0}
\newcommand{\ddo}{d_0}

\newcommand{\pdo}{pd_0}
\newcommand{\ulxo}{\underline{x_0}}
\newcommand{\ulpdo}{\underline{pd_0}}

\title{Power domination reconfiguration}
\author{Beth Bjorkman\thanks{Air Force Research Laboratory, 2241 Avionics Circle, Wright-Patterson Air Force Base, OH 45433, USA (beth.morrison@us.af.mil)}\and Chassidy Bozeman\thanks{Department of Mathematics and Statistics, Mount Holyoke College, South Hadley, MA 01075, USA (cbozeman@mtholyoke.edu)}\and Daniela Ferrero\thanks{Department of Mathematics, Texas State University, San Marcos, TX 78666, USA (dferrero@txstate.edu).}\and Mary Flagg\thanks{Department of Mathematics, Statistics and Computer Science, University of St. Thomas, Houston, TX 77006, USA (flaggm@stthom.edu).}\and Cheryl Grood\thanks{Department of Mathematics and Statistics, Swarthmore College, Swarthmore, PA 19081, USA (cgrood1@swarthmore.edu)}\and Leslie Hogben\thanks{Department of Mathematics, Iowa State University,
Ames, IA 50011, USA and American Institute of Mathematics, 600 E. Brokaw Road, San Jos\'e, CA 95112, USA
(hogben@aimath.org).}\and Bonnie Jacob\thanks{Science and Mathematics Department, National Technical Institute for the Deaf, Rochester Institute of Technology, Rochester, NY 14623, USA (bcjntm@rit.edu)} \and Carolyn Reinhart\thanks{Department of Mathematics and Statistics, Swarthmore College, Swarthmore, PA 19081, USA (carolynreinhart196@gmail.com)}}

\begin{document}

\maketitle

\vspace{-10pt}
\begin{abstract}  The study of token addition and removal and  token jumping reconfiguration graphs for power domination is initiated.  Some results established here can be extended by applying  the  methods used for power domination to reconfiguration graphs for other parameters such as domination and zero forcing, so these results are first established in a universal framework. 
\end{abstract}

\noi {\bf Keywords} reconfiguration; power domination;  token addition and removal; token jumping; hypercubes

\noi{\bf AMS subject classification} 05C69, 68R10, 05C57 

\section{Introduction}\label{s:intro}

Reconfiguration examines relationships among solutions to a problem. The \emph{reconfiguration graph} has as its vertices solutions to a problem and edges are determined by a \emph{reconfiguration rule} that describes  relationships between the solutions. Reconfiguration can also be viewed as a transformation process, where the reconfiguration rule describes a single step between one solution and another. Then, reconfiguration is a sequence of steps between solutions in which each intermediate state is also a solution.  Being able to transform one solution to another another is equivalent to having a path between the two solutions in the reconfiguration graph, i.e., the two solutions are in the same connected component of the reconfiguration graph.  

Nishimura describes three types of reconfiguration rules, token addition and removal, token jumping, and token sliding in \cite{N18}.  There she surveys recent work on both structural and algorithmic questions   across a broad variety of parameters. For many reconfiguration problems arising from  graphs (including all those we study), a solution can be represented as a subset of the vertices of a graph $G$ (such a subset is sometimes identified by the placement of a token on each of the vertices of the subset). The \emph{token addition and removal (TAR) reconfiguration graph} has an edge between two sets  if and only if one can be obtained from the other by the addition or removal of a single vertex.  The \emph{token jumping (TJ) reconfiguration graph} has an edge between two sets  if and only if one can be obtained from the other by exchanging a single vertex.  The \emph{token sliding (TS) reconfiguration graph} has an edge between two sets  if and only if one can be obtained from the other by exchanging a single vertex and the vertices exchanged are adjacent in $G$. Both the TAR graph (involving all feasible subsets of vertices of $G$) and the $k$-TAR graph, which permits feasible subsets of at most $k$ vertices, are studied. 

 In this paper we  initiate the study of the token addition and removal (both TAR and $k$-TAR) and the token jumping reconfiguration graphs for power domination.  The  concept of power domination was introduced in graph theory to model the monitoring process of electrical power networks  in \cite{HHHH02} 
 (the formal definitions of power domination and related terms are given below  after further discussion).  An electric power network can be modeled by a graph in which a vertex can represent a power plant, station, transformer, generator, or user, and edges represent the power lines between them.  An electric power network must be monitored continuously in order to prevent blackouts and power surges, and the most widely used monitoring method consists of placing Phasor Measurement Units (PMUs) at selected network locations. PMUs measure  the magnitude and the phase angle of the electric waves at the location where they are placed \cite{BMBA93}. Due to their cost, it is important to minimize the number of PMUs used to monitor a power grid. The {\em PMU placement problem} in electrical engineering consists of finding  locations in the network where PMUs can be placed to monitor the entire network with the minimum number of PMUs possible.   However, the implementation of large scale systems of PMUs has shown that minimizing the number of PMUs yields unsatisfactory results due to the lack of redundancy in the event of failures. Moreover, research in electrical engineering has proven that the addition of a few PMUs is a cost effective alternative to enhance the system reliability \cite{PB18,SW18}. As a result, there is now interest on non-minimum solutions to the PMU placement problem, and in particular, in the transformation of existing solutions into other more reliable alternatives \cite{SW18}. 

 There are certain similarities with previous work on TAR and $k$-TAR reconfiguration graphs for domination \cite{HS14} and TJ reconfiguration graphs for zero forcing \cite{GHH} (there called `token exchange' reconfiguration graphs).  This led us to establish many of the standard properties in a universal framework for properties and parameters satisfying certain conditions; examples to which these results apply include power domination, domination, and zero forcing.  TAR and $k$-TAR  reconfiguration  graphs are discussed in Section \ref{s:universal} and TJ  reconfiguration  graphs in Section \ref{s:universal-TE}.  Power domination  TAR and $k$-TAR reconfiguration  graphs are discussed in Section \ref{s:TAR} and power domination TJ  reconfiguration  graphs in Section \ref{s:TE}.  Each of these sections contains a precise definition of the relevant reconfiguration  graph and related terms.

In Section \ref{s:universal} we establish properties of a graph $G$ that are encoded in its TAR reconfiguration graph, such as the order of $G$ and $X(G)$ (where $X(G)$ is a parameter satisfying the conditions used to establish universal results).  In Section \ref{s:TAR} we present a family of graphs $G_n$ whose 
TAR graphs require a large $k$  to ensure the $k$-TAR graph is connected (arbitrarily exceeding the standard lower bound).  We also show that  the star is the only connected graph that achieves maximum upper power domination number (largest  minimal power dominating set) and thus the star is then only graph without isolated vertices that realizes its  TAR reconfiguration graph.  In Section \ref{s:TE} we provide techniques for realizing specific graphs as power domination TJ  reconfiguration  graphs, and give examples of disconnected TJ reconfiguration graphs.

 
In the remainder of this introduction we define additional terminology and notation. A \emph{graph} is a pair $G=(V(G),E(G))$ where $V(G)$ is a non-empty set of vertices and $E(G)$ is a set of edges; an \emph{edge} is a $2$-element subset of $V(G)$. The \emph{order} and \emph{size} of $G$ are  defined as $|V(G)|$ and $|E(G)|$, respectively. Let $u$ and $v$ be vertices of $G$. If $e=\{u,v\}$ is an edge of $G$ ($e$ also denoted by $uv$ or $vu$), then $e$ is \emph{incident} with $u$ and $v$, and in this case vertices $u$ and $v$ are said to be \emph{adjacent} or \emph{neighbors}. The \emph{(open) neighborhood} of $v$ is $N_G(v)$ defined as the set of all neighbors of $v$ and the \emph{closed neighborhood} of $v$ is $N_G[v]$ defined as $N_G[v]=N_G(v)\cup \{v\}$. The cardinality of $N_G(v)$ defines the \emph{degree} of $v$, which is denoted by $\deg_G(v)$. The \emph{maximum degree} and \emph{minimum degree} of $G$ are  denoted by $\Delta(G)$ and $\delta(G)$,  and they are defined as the maximum and minimum of the set of degrees of all vertices of $G$, respectively. A graph $G$ is \emph{regular} when $\Delta(G)=\delta(G)$, and in this case, $G$ is \emph{$r$-regular} for the integer $r$ such that $\Delta(G)=r$ and $\delta(G)=r$.  A \emph{leaf} is a vertex of degree one. The \emph{subgraph induced by $W\subseteq V(G)$}, denoted by $G[W]$, has the vertex set $W$ and edge set $\{wz: w,z\in W\mbox{ and }wz\in E(G)\}$.

A \emph{path} of length $\ell$  is a sequence of distinct vertices $v_0,\ldots, v_{\ell}$ such that 
$v_i$ is a neighbor of $v_{i+1}$ for every integer $i$, $0\leq i\leq \ell -1$. Such a path is also described as a path between  $v_0$ and $v_\ell$ (or vice versa).  A graph $G$ is \emph{connected} if there exists a path between any two different vertices of $G$. If $u$ and $v$ are different vertices in a connected graph $G$, the \emph{distance} between $u$ and $v$, denoted by $\dist_G(u,v)$, is defined to be the minimum length of a $uv$-path, or a path between $u$ and $v$. The \emph{diameter} of a connected graph $G$ is $\diam(G)$ defined as the maximum value of $\dist_G(u,v)$ over all pairs of different vertices $u$ and $v$ of $G$.  Two graphs $G$ and $H$ are \emph{disjoint} if $V(G)\cap V(H)=\emptyset$ and the \emph{disjoint union} of $G$ and $H$ is denoted by $G\sqcup H$.

For any positive integer $n$, let $V_n=\{v_1,\ldots, v_n\}$. A \emph{path of order $n$} is a denoted by $P_n$ and defined by $V(P_n)=V_n$  and $E(P_n)=\{\{v_i,v_{i+1}\}: 1\leq i\leq n-1\}$. A \emph{complete graph of order $n$} is denoted by $K_n$ and defined to be the $(n-1)$-regular graph with $V(K_n)=V_n$. For $n\geq 3$, a \emph{cycle of order $n$} is denoted $C_n$ and defined by $V(C_n)=V_n$ and $E(C_n)=E(P_n)\cup \{\{v_n,v_1\} \}$. For $n\geq 4$, the \emph{wheel of order $n$} is denoted by $W_n$ and is the graph defined by $V(W_n)=V(C_{n-1})\cup \{v_0\}$ and $E(W_n)=E(C_{n-1})\cup \{\{v_0,v_i\}: 1\leq i \leq n-1\}$.

When studying reconfiguration graphs we often use the \emph{symmetric  difference of sets $S$ and $S'$}, which we denote by   $S \ominus S'=(S\setminus S')\cup (S'\setminus S) =S \cup S' - S \cap S'$.


Let $G$ be a graph and let $S$ be a non-empty subset of vertices of $G$. For a non-negative integer $i$, recursively define $P^i[S]$ by 
\begin{itemize}
\item[1.] $P^0[S]=S$ and $P^1[S]=N[S]$ 
\item[2.] If $i\geq 2$, then $P^i[S]=P^{i-1}[S]\cup \{u: N_G(v)\setminus P^{i-1}[S] =\{u\} \mbox{ for some } v\in P^{i-1}[S] \}$.
\end{itemize}

If there exists $t$ such that $P^{t}[S]= V(G)$ then $S$ is a \emph{power dominating set} of $G$. 
In particular, if $P^{1}[S]= V(G)$ then $S$ is a \emph{dominating set} of $G$.  A power dominating set with minimum cardinality is a \emph{minimum power dominating set}. 
The \emph{power domination number} of $G$, denoted by $\pd(G)$, is the cardinality of a minimum power dominating set of $G$. 
The definition of power domination presented here is a simplified version of that in \cite{HHHH02} that was shown to be equivalent by Brueni and Heath in \cite{BH}.
For every positive integer $i$, $P^i(S)$ denotes the set of vertices observed from $S$ in round $i$ and it is defined by $P^i(S)=P^i[S]\setminus P^{i-1}[S]$. For $i\geq 2$, the definition of $P^i[S]$ implies that for every $u$ in $P^i(S)$ there exists $v$ in $P^{i-1}[S]$ such that $N_G(v)\setminus P^{i-1}[S] =\{u\}$, and in this case, we say that $v$ observes $u$ in round $i$.

\section{A universal approach to TAR reconfiguration}\label{s:universal}

 Here we establish some basic properties of token addition and removal (TAR) reconfiguration graphs  for  a graph parameter $X(G)$ defined as the minimum cardinality among all subsets of vertices of $G$ that satisfy a given property $X$, each of which is called an $X$-set.
A {\em minimal $X$-set} is an $X$-set such that removal of any vertex from $X$ results in a set that is not an $X$-set; every minimum $X$-set is a minimal $X$-set but not conversely. The {\em upper $X$ number} $\xbar(G)$ is the maximum cardinality of a {minimal} $X$-set. The next definition describes properties of $X$ assumed when discussing TAR reconfiguration.

\begin{defn}\label{d:Xprop} Let $X$ be a property  of sets of vertices of a graph such that:
\ben[(1)]
\item\label{x1} If $S$ is an $X$-set and $S\subset S'$, then $S'$ is an $X$-set.
\item\label{x2}  $\emptyset$ is not an $X$-set.
\item\label{d-c:nK1} An $X$-set of a disconnected graph is the union of an $X$-set of each component. 
\item\label{x4} For a graph $G$ of order $n$ with no isolated vertices, every set of $n-1$ vertices is an $X$-set. 
\een
\end{defn}

 Examples of properties satisfying the conditions in Definition \ref{d:Xprop} and their corresponding graph parameters include:
 
\bit
\item An $X$-set is a dominating set and $X(G)$ is the domination number $\gamma(G)$.
\item An $X$-set is a power dominating set and $X(G)$ is   the power domination number $\pd(G)$.
\item An $X$-set is a zero forcing set and $X(G)$ is  the zero forcing number $\Z(G)$.
\item An $X$-set is a positive semidefinite zero forcing zero forcing set and $X(G)$ is  the positive semidefinite zero forcing  number $\Zp(G)$.\eit



\begin{defn} The {\em $k$-token addition and removal (TAR) reconfiguration graph for a parameter $X$}, denoted by $\xtar_k(G)$, has as its vertex set the set of all  $X$-sets of cardinality at most $k$. There is an edge joining the two vertices  $S_1$ and $S_2$ of $\xtar_k(G)$ if and only if one can be obtained from the other by the addition or removal of a single vertex, i.e., $|S_1 \ominus S_2|=1$.  Furthermore, $\xtar (G)=\xtar_n(G)$ for a graph $G$ of order $n$.
\end{defn}



 
The domination TAR reconfiguration graph and the upper domination number are discussed in \cite{HS14}. The power domination TAR reconfiguration graph and the upper power domination number are discussed in Section \ref{s:TAR}. 

\begin{rem}\label{no-isol-vtx} Let $X$ be a property satisfying the conditions in Definition \ref{d:Xprop}. Any isolated vertex of a graph $G$ must be in every $X$-set of $G$.  A graph with no isolated vertices has more than one minimal $X$-set, and no vertex is in every minimal $X$-set.
Suppose $G'=G\sqcup rK_1$ and $G$ has no isolated vertices.   In view of     Definition \ref{d:Xprop}\eqref{d-c:nK1},   $X(G')=X(G)+r$, $\xtar_{k+r}(G')\cong\xtar_k(G)$, and $\xtar(G')\cong\xtar(G)$. 
Therefore, it is sufficient to study reconfiguration   graphs with no isolated vertices. This is particularly important from the algorithmic point of view, as removing all isolated vertices prior to running an algorithm could significantly reduce the complexity. 
\end{rem}

\begin{obs}\label{new-obs} If $X$ is a property satisfying   Definition \ref{d:Xprop}(\ref{x1}), then  knowledge of all the minimal $X$-sets is sufficient to determine $\xtar (G)$ and $\xtar_k (G)$ for all $k\ge X(G)$.  \end{obs}

\subsection{Basic properties}\label{ss:ubasic}

In this section we present some elementary properties of TAR reconfiguration graphs. 


\begin{prop}\label{x-deg-diam} Let $X$ be a property satisfying the conditions in Definition \ref{d:Xprop} and
let  $G$ be a graph on $n$ vertices with no isolated vertices. Then 
\ben\item\label{p-c:umaxdeg} 
$\Delta(\xtar (G))=n$.
\item\label{p-c:udist} For any two $X$-sets $S$ and $S'$,  $\dist_{\xtar(G)}(S,S')= |S\ominus S'|$.
\item\label{p-c:udiam} $\diam(\xtar (G))=\max_{S,S'}(|S\ominus S'|)$ where $S$ and $S'$ are  $X$-sets.
\item\label{p-c:udiam-bd} $n-X(G)\le \diam(\xtar (G))\le n$.
\een
\end{prop}


\bpf \eqref{p-c:umaxdeg}:  
Let $S\subseteq V(G)$ be a vertex of $\xtar(G)$ and let $|S|=r$. Then there are at most $r$ vertices that can be deleted (one at a time) leaving an $X$-set, and at most $n-r$ vertices that can be added (one at a time).  Thus $\deg_{\xtar(G)}(S)\le n$.  Since every set of $n-1$ vertices is an $X$-set, $\deg_{\xtar(G)}(V(G))= n$.

\eqref{p-c:udist} and \eqref{p-c:udiam}: A path of length $|S\ominus S'|$ from $S'$ to $S$ can be realized by successively adding vertices in $S\setminus S'$ and then deleting vertices in $ S'\setminus S$, and there is no shorter path.  

\eqref{p-c:udiam-bd}: Apply \eqref{p-c:udist}, so $\diam(\xtar(G))\leq n$ is immediate. For $n-X(G)\le \diam(\xtar(G))$,  consider the distance between a minimum $X$-set and $V(G)$.   
\epf

 \begin{prop}\label{p:umindegreeTAR}
Let $X$ be a property satisfying the conditions in Definition \ref{d:Xprop} and let $G$ be a graph on $n\geq3$ vertices. 
Then  $\delta(\xtar(G))=n-\xbar(G)$.
\end{prop}

\bpf
Let 
$R=\xtar(G)$.   Let $M$ be a minimal $X$-set of $G$, with $|M|=m$ for some $m\geq 1$.  Since $N_R(M)=\{M \cup \{v\}: v \in V(G)\setminus M\}$,  $n-m=\deg_R(M)\ge \delta(R)$. 

Let $S$ be an $X$-set such that $M \subsetneq S$. Denote the elements of $S\setminus M$ by $U= \{u_1,\ldots,u_\ell\} $ and those of  $V(G)\setminus S$ by $W=\{w_1,\ldots, w_{n-m-\ell} \}$. Then the sets of the form $S-\{u_i\}$ with $u_i \in U$ and $S \cup  \{w_j\}$ with $w_j \in W$ are neighbors of $S$ in $R=\xtar(G)$, and $|N_R(S)| \geq n-m$. 

Since $\deg_R(M) = n-|M|$ for any minimal power dominating set, a minimal power dominating set of maximum cardinality is a minimum degree vertex in $R$.  That is, $\delta(\xtar(G)) = n - \xbar(G)$.
\epf

\begin{rem}\label{inducedstar}
Let $X$ be a property satisfying the conditions in Definition \ref{d:Xprop} and let $G$ be a graph of order $n$ with no isolated vertices. By considering the subgraph of $\xtar(G)$ induced by $V(G)$ and the $n$ sets of cardinality $n-1$, we see that $\xtar(G)$ has an induced subgraph isomorphic to $K_{1,n}$.

Let $S$ be a minimum $X$-set of $G$. For every vertex $v_1,\dots,v_{n-X(G)}\in V(G)\setminus S$, the set $S\cup\{v_i\}$ is an $X$-set of $G$ and is adjacent to $S$ in $\xtar_{X(G)+1}(G)$. 
Since in $\xtar_{X(G)+1}(G)$ there are no edges between $X$-sets of the same cardinality, the subgraph of $\xtar_{X(G)+1}(G)$  induced by $S$ and $S\cup \{v_i\}, i=1,\dots,n-X(G)$ is isomorphic to $K_{1,n-X(G)}$.  
\end{rem}


\subsection{Connectedness}
A main question in reconfiguration is: For which $k$ is the $k$-reconfiguration graph connected?   As noted in \cite{HS14} for domination, $\xtar (G)$ is always connected (every $X$-set can be augmented one vertex at a time to get $V(G)$).  Following \cite{HS14}, we consider the least $k_0$ such that $\xtar_k(G)$ is connected for all $k\ge k_0$; this value is denoted by $\xxo(G)$.  As suggested in \cite{N18}, we also consider the least $k$ such that $\xtar_k(G)$ is connected; this value is denoted by $\ulxo(G)$.

\begin{prop}\label{basic1}
Let $X$ be a property satisfying the conditions in Definition \ref{d:Xprop} and let $G$ be a graph of order $n$ with no isolated vertices.  Then $X(G)\le \ulxo(G)\le\xxo(G)$ and $\xbar(G)+1\le \xxo(G)\le\min\{\xbar(G)+X(G),n\}$.  If $G$ has more than one minimum $X$-set, then $X(G)+1\le \ulxo(G)$.
\end{prop}
\bpf  From the definitions, $X(G)\le \ulxo(G)\le  \xxo(G)$ is immediate. Recall that by our assumptions about $X$, $G$ has more than one minimal $X$-set. Let $S\subset V(G)$ be minimal $X$-set with $|S|=\xbar(G)$.  Then $S$ is an isolated vertex of $\xtar_{\xbar(G)}(G)$ (because we can't add a vertex, and removal results in a set that is not an $X$-set).  Thus $\xbar(G)+1\le \xxo(G)$.  Similarly, $X(G)+1\le \ulxo(G)$ when $G$ has more than one minimum $X$-set.


Let $k_0=\min\{\xbar(G)+X(G),n\}$.  Let $S\subset V(G)$ be a minimal $X$-set of $G$ and $S'\subset V(G)$ be a minimum $X$-set of $G$.    To ensure $\xtar_k(G)$ is connected for all $k\ge k_0$, it is sufficient to show that  every such pair of vertices $S$ and $S'$ is connected in $\xtar_{k_0}(G)$, because building onto minimal sets does not disconnect the reconfiguration graph. 
If $\xbar(G)+X(G)\le n$, define $S''=S\cup S'$.  Then each of $S$ and $S'$ is connected by a path to $S''$ by adding one vertex at a time.  Thus $\xxo(G)\le\xbar(G)+X(G)$. 
\epf

The bounds in Proposition \ref{basic1} are already known for some  parameters, e.g., for domination,
$\gbar(G)+1\le \ddo(G)\le\min\{\gbar(G)+\gamma(G), n\}$ \cite{HS14} (where   $\ddo(G)$  denotes the least $k_0$ such that the domination $k$-TAR reconfiguration graph is connected for all $k\ge k_0$). 

\begin{cor}\label{c:X1}
Let $X$ be a property satisfying the conditions in Definition \ref{d:Xprop} and let $G$ be a graph with no isolated vertices. If $X(G)=1$, then $\xxo(G)=\xbar(G)+1$ . 
\end{cor}

 The next remark presents a condition that guarantees that a graph $G$ with $X(G)>1$ satisfies $\xbar(G)+2\le \xxo(G)$.

\begin{rem}
Let $X$ be a property satisfying the conditions in Definition \ref{d:Xprop} and let $G$ be a graph with no isolated vertices and $X(G)>1$. If $G$ has an $X$-set $S$ such that $|S|=\xbar(G)$ and for every $u$ and $v$ such that $u\in S$ and  $v\in V(G) \setminus S$, the set $B=(S\setminus \{u\} )\cup \{v\}$ is not an $X$-set of $G$, then   the graph $\xtar_{\xbar(G)+1}(G)$ has a connected component isomorphic to $K_{1,n-\xbar(G)}$ where $n= |V(G)|$ (cf. Remark \ref{inducedstar}).  Since there is a minimal $X$-set that is not $S$, there is another component.
\end{rem}

\begin{prop}\label{p:disj-min-sets}
Let $X$ be a property satisfying the conditions in Definition \ref{d:Xprop} and let  $G$ be a graph. Partition the minimal $X$-sets  into two sets $\{S_1,\dots,S_k\}$ and $\{T_1,\dots,T_\ell\}$.   Let $s=\max_{i=1}^k|S_i|$, $t=\max_{i=1}^\ell|T_i|$, 
$\mcs=\cup _{i=1}^k S_i$,    $\mct=\cup _{i=1}^\ell T_i$, and suppose $\mcs\cap\mct=\emptyset$. Then 
\ben
\item $\diam(\xtar_{s+t}(G))\ge s+t$.
\item $x_0(G)\ge s+t$.
\een
\end{prop}
\bpf 
Let $S \in \{S_1,\dots,S_k\}$ with $|S|=s$, and $T \in \{T_1,\dots,T_\ell\}$ with $|T|=t$.   Since $S\cap T=\emptyset$,  $\dist_{\xtar(G)}(S,T)=s+t$ 
by Proposition \ref{x-deg-diam}.  
Thus $\diam(\xtar_{s+t}(G))\ge s+t$.  Furthermore, there is no path between $S$ and $T$ in $\xtar_k(G)$ for $k$ with $\max\{s,t\}\le k\le s+t-1$.
\epf

Note that the assumption that $\mcs\cap\mct=\emptyset$ in Proposition \ref{p:disj-min-sets} implies that any graph to which this applies has no isolated vertices. 


\subsection{Hypercube representation}

 The graph having all  $(0,1)$-sequences as vertices with two sequences adjacent exactly when they differ in one place is a characterization of $Q_n$, the hypercube of dimension $n$. There is a well-known representation of any TAR reconfiguration graph as a subgraph of a hypercube.  Let $G$ be a graph with $|V(G)|=n$. Any subset $S$ of $V(G)$ can be represented by a sequence $(s_1, \dots, s_n)$ where $s_i=1$ if $v_i\in S$ and $s_i=0$ if $v_i\not\in S$.    
If we restrict  $S$ to being an $X$-set,  then $\xtar(G)$ is isomorphic to a subgraph of $Q_n$.  

\begin{rem}\label{r:hypercube-bip} Let $X$ be a property satisfying the conditions in Definition \ref{d:Xprop}. For every graph $G$  of order $n$ and $k\ge X(G)$, $\xtar_k (G)$ is a subgraph of $Q_n$.  Thus $\xtar_k (G)$  is bipartite. Hence, a graph that is not bipartite cannot be realized as a TAR graph.
\end{rem}

\begin{rem} Let $X$ be a property satisfying the conditions in Definition \ref{d:Xprop}. It is easy to see that  $\xtar (G)$ is not isomorphic to any hypercube for any graph  of order $n\ge 2$:  Suppose first that $G$ has  no isolated vertices. Recall that $\Delta(\xtar(G))=n$, so if $\xtar(G)$ is a hypercube it must be $Q_n$. But the order of $Q_n$ is $2^n$ and the order of  $\xtar (G)$ is at most $2^n-1$ since $\emptyset$ is not an $X$-set. If $G'=G\sqcup rK_1$ and $G$ has no isolated vertices, then $\xtar(G')\cong\xtar(G)\not\cong Q_d$ for any $d$. 
 \end{rem}

\begin{lem}\label{L:uembedQt} Let $X$ be a property satisfying the conditions in Definition \ref{d:Xprop}.
Let $G$ be a graph on $n\geq 2$ vertices 
and let $1\leq t \leq n-1$. Then   $X(G)\leq n-t$ if and only if $\xtar(G)$ has an induced subgraph isomorphic to the hypercube $Q_t$. 
\end{lem}
\bpf 
Let $S$ be a minimum $X$-set and let $W=V(G)\setminus S$.  The induced subgraph of $\xtar(G)$ having vertices consisting of sets of the form $S\sqcup W'$ over all subsets $W'\subseteq W$  is $Q_{n-X(G)}$.  Any hypercube $Q_{t}$ for $n-X(G)\ge t\ge 1$ is an induced subgraph of $Q_{n-X(G)}$.

 Suppose $H$ is an induced subgraph of $\xtar(G)$ isomorphic to $Q_t$ for some $1\leq t \leq n-1$. Choose $S \in V(H)$ such that $|S|$ is minimum over all $X$-sets in $H$.  Since no vertex in $H$ has fewer vertices than $S$, every one of the $t$ neighbors of $S$ in $H$ is obtained by adding a vertex of $G$ to $S$.  Thus $|V(G)\setminus S| \ge t$ and $X(G)\le |S|\le n-t$.
\epf

\begin{cor}\label{c:X-Qt}Let $X$ be a property satisfying the conditions in Definition \ref{d:Xprop}.
Let $G$ be a graph on $n\geq 2$ vertices. 
Then $d=n-X(G)$ is the  maximum dimension of a hypercube isomorphic to an induced subgraph of the reconfiguration graph 
 $\xtar(G)$. 
\end{cor}

\begin{cor}\label{c:u-iso-props}Let $X$ be a property satisfying the conditions in Definition \ref{d:Xprop}.
Suppose $G$ and $H$ have no isolated vertices and $\xtar(G)\cong\xtar(H)$. Then
\ben
\item $|V(G)|=|V(H)|$.
\item $X(G)=X(H)$.
\item $\xbar(G)=\xbar(H)$. 

\een
\end{cor}
\bpf By Proposition \ref{x-deg-diam}, $|V(G)|=\Delta(\xtar(G))=\Delta(\xtar(H))=|V(H)|$. 
Corollary \ref{c:X-Qt} implies $X(G)=X(H)$. By Proposition \ref{p:umindegreeTAR}, $\xbar(G)=n-\delta(\xtar(G))=n-\delta(\xtar(H))=\xbar(H)$. 
\epf

\section{Power domination token addition and removal (TAR) graphs}\label{s:TAR}

 In this section we establish results for the $k$-token addition and removal (TAR) reconfiguration graphs for power domination.  The {\em power domination number}  $\pd(G)$ is the cardinality of a {minimum} power dominating set, and the {\em upper power domination number}  $\pdbar(G)$ is the maximum cardinality of a {minimal} power dominating set.  

\begin{defn} The {\em $k$-token addition and removal (TAR) reconfiguration graph for power domination}, denoted by $\tar_k(G)$ uses   power dominating sets of cardinality at most $k$ as the vertices; $\tar (G)=\tar_n(G)$ for a graph $G$ of order $n$.  These graphs have an edge between two vertex sets $S_1$ and $S_2$ when one can be obtained from the other by the addition or removal of a single vertex, i.e., $|S_1 \ominus S_2|=1$.  
\end{defn}

Observe that power domination satisfies the conditions in Definition \ref{d:Xprop}.

\begin{ex} If $G$ is a graph such that any one of its vertices is a power dominating set, $\tar(G)$ is isomorphic to an $n$-dimensional hypercube with one vertex deleted  because the only subset of $V(G)$ that is not a power dominating set is the empty set.  Examples of such graphs include $C_n, P_n, K_n,$ and $W_n$.
\end{ex}

\subsection{Basic properties}\label{ss:basic}


 
Remark \ref{no-isol-vtx} applies to power domination, so we focus our attention on graphs with no isolated vertices.  The next result applies Corollary \ref{c:u-iso-props} to power domination.

\begin{cor}\label{c:pd-iso-props}
Suppose $G$ and $H$ have no isolated vertices and $\tar(G)\cong\tar(H)$. Then
\ben
\item $|V(G)|=|V(H)|$.
\item $\pd(G)=\pd(H)$.
\item $\pdbar(G)=\pdbar(H)$. 
\een
\end{cor}
The next result applies Proposition  \ref{x-deg-diam} \eqref{p-c:umaxdeg} and \eqref{p-c:udist} to power domination.
\begin{cor}\label{pd-deg-dist} 
Let  $G$ be a graph on $n$ vertices with no isolated vertices. Then 
\ben\item 
$\Delta(\tar (G))=n$.
\item In $\tar(G)$, the distance between any two power dominating sets $S$ and $S'$ is exactly $|S\ominus  S'|$.
\een
\end{cor}

\begin{prop}\label{lem:neighbors} Let $G$ be a graph with no isolated vertices, let $S$ be a minimal power dominating set of $G$, and define $S'=V(G)\setminus S$. Then $S\subseteq N(S') $, $S'$  is a dominating set,  and $G$  contains two disjoint minimal power dominating sets.
\end{prop}
\bpf  Since $S$ is minimal and $G$ has no isolated vertices, each vertex in $S$ must have a neighbor in $S'$. Thus $S'$ is a dominating set, which must contain a minimal power dominating set (that is disjoint from $S$).
\epf

We can now improve the bound on the diameter given in  Proposition \ref{x-deg-diam}\eqref{p-c:udiam-bd} for the power domination TAR reconfiguration graph.
\begin{prop}\label{lem:diameterbounds}
Let $G$ be a graph of order $n$ with no isolated vertices. Then $\diam(\tar(G))=n$.
\end{prop} 

\begin{proof}  Let $S$ be a minimal power dominating set of $G$.  Since $S'=V(G)\setminus S$  is a power dominating set by Proposition \ref{lem:neighbors}, the distance between $S$ and $S'$ in $\tar(G)$ is $|S\cup S'|-|S\cap S'|=n$, and thus $\diam(\tar(G))=n.$
\end{proof}

The next result applies Proposition \ref{p:umindegreeTAR} to power domination.
\begin{cor}\label{p:mindegreeTAR}
Let $G$ be a graph on $n\geq3$ vertices 
Then  $\delta(\tar(G))=n-\pdbar(G)$.
\end{cor}


\subsection{Connectedness}

The least $k_0$ such that $\tar_k(G)$ is connected for all $k\ge k_0$ is denoted by $\pdo(G)$, and the least $k$ such that $\tar_k(G)$ is connected is denoted by $\ulpdo(G)$.
The next two results are applications of Propositions \ref{basic1} and \ref{c:X1} to power domination.

\begin{cor}\label{basic1-pd} 
Let $G$ be a graph of order $n$ with no isolated vertices. Then $\pd(G)+1\le \ulpdo(G)\le\pdo(G)$ and $\pdbar(G)+1\le \pdo(G)\le\min\{\pdbar(G)+\pd(G),n\}.$  
\end{cor}

\begin{cor}\label{c:pd1}
Let $G$ be a graph with no isolated vertices. If $\pd(G)=1$, then $\pdo(G)=\pdbar(G)+1$ . 
\end{cor}

There are many examples that achieve the lower bound $\pdo(G)=\pdbar(G)+1$, including when   $\pdbar(G)+1< \min\{\pdbar(G)+\pd(G),n\}.$  Here we give one example.



\begin{ex}\label{TAR:compbipart} Let $b\ge a\ge 3$.  Then $\pd(K_{a,b})=2$ and $\pdbar(K_{a,b})=b-1\ge 2$, since there are three types of minimal power dominating sets of $K_{a,b}$: any subset of $a-1$ vertices from the partite set $A$ of size $a$; any subset of $b-1$ vertices from the partite set $B$ of size $b$; and any subset consisting of one vertex from $A$ and one from $B$.  
To see that  $\pdo(K_{a,b})=b=\pdbar(K_{a,b})+1$, which is less than the upper bound of $\pdbar(K_{a,b})+\pd(K_{a,b}) = \pdbar(K_{a,b})+2$, it suffices to show that there is a path in $\tar_b (K_{a,b})$ from any minimum power dominating set to any minimal power dominating set.  Let $S$ be a minimum power dominating set.  Thus $S = \{v,w\}$ for some $v \in A$ and $w \in B$.  Let $M$ be a minimal power dominating set. If $S\cap M\ne \emptyset$, then $|S\cup M|\le b$ and thus there is a path in $\tar_b (K_{a,b})$ from $S$ to $M$. So assume that $M \cap S = \emptyset$. For each of the three types of minimal power dominating sets, we give a sequence of adjacent vertices in $\tar_b (K_{a,b})$ that defines a path from $S$ to $M$.  \\[1mm]
 \underline{Case 1:  $M = \{\hat{v},\hat{w}\}$:}\vspace{-6pt}
 \[S_0=S, \  S_1 =\{v,w,\hat v\}= S_0 \cup \{\hat{v}\},  \ S_2 =\{w,\hat v\} = S_1 \setminus \{v\}, \  S_3=\{w,\hat v,\hat w\}=S_2 \cup \{\hat{w}\}, \  S'''\setminus \{w\} = M.\vspace{-6pt}\]
 Since each subset $S_i$ of $V(K_{a,b})$ above has at most 3 vertices,  each subset is a vertex in $\tar_b (K_{a,b})$. \\[1mm]
  \underline{Case 2:  $M = \{v_1, \ldots, v_{a-1}\} \subset A$:}\vspace{-4pt}
 \[S_0=S, \ S_1 = S_0 \cup \{v_1\}, \ S_2 = S_1 \setminus \{v\}, \ S_{i} = S_{i-1} \cup \{v_{i-1}\} \mbox{ for } i=3, \ldots, a, \  S_{a+1}=S_{a} \setminus \{w\} = M.\vspace{-6pt}\]   
 In other words, we first add a vertex from $M$ to $S$, then remove $v$, then add the remaining vertices in $M$ one at a time to $S$, and finally remove $w$ to reach $M$.  Each subset of vertices we create in this way has no more than $a$ vertices, and thus each subset is a vertex in $\tar_b (K_{a,b})$. \\[1mm]
   \underline{Case 3:  $M = \{w_1, \ldots, w_{b-1}\} \subset B$:} This case is similar to Case 2. 
\end{ex}

Next we present a family of of graphs that realize the upper bound $\pdo(G)\le\min\{\pdbar(G)+\pd(G), n\}$ and have $\pdbar(G)+1<\min\{\pdbar(G)+\pd(G), n\}$.

\begin{defn}\label{e:bigupper}
For any integer $n\geq 3$ let $G_n=(V_n,E_n)$ be the graph defined by
\begin{itemize}
\item $V_n=T_n\cup \big( \cup_{i=1}^n S_{n,i}\big)$ where $T_n=\{u_1,\ldots, u_{n-1}\}$ and $S_{n,i}=\{v^i_1,\ldots , v^i_n\}$, for every integer $i$,  $1\leq i\leq n$.
\item For integers $i=1,\ldots, n$ and $j=1,\ldots , n-1$: \hskip 0.5cm $N_{G_n}(v^i_j)=\{u_j\}\cup \{v^i_p: 1\leq p\leq n, p\not=j\}$
\subitem \hskip 6.9cm $N_{G_n}(v^i_n)=\{v^i_1,\ldots , v^i_{n-1}\}$
\subitem \hskip 6.9cm $N_{G_n}(u_j)=\{v^r_j: 1\leq r\leq n\}$
\end{itemize}
\end{defn}

The graph $G_4$ is shown in Figure \ref{fig:TARdisconnect}.

\begin{figure}[h!]
    \centering
  
        \scalebox{.4}{\includegraphics{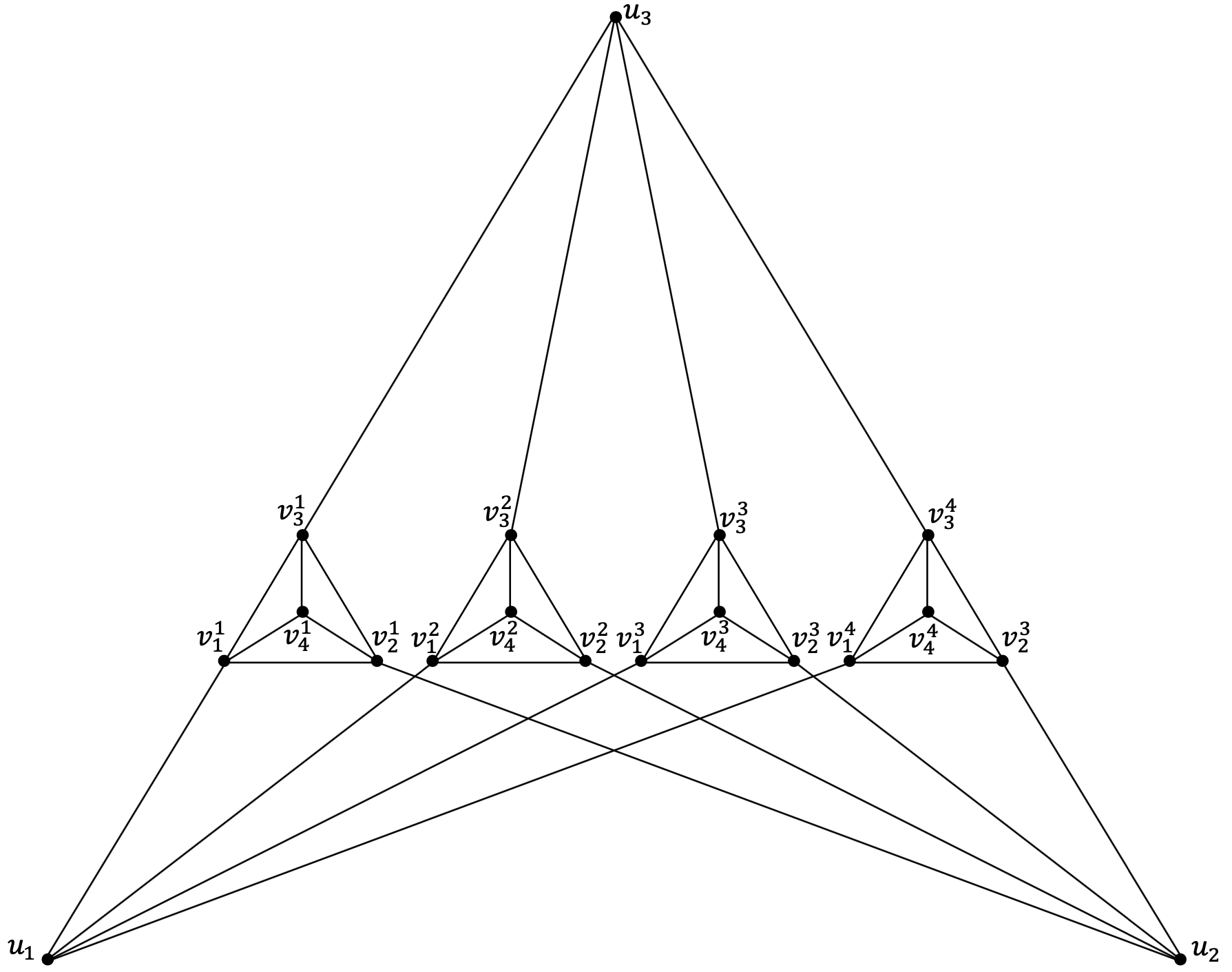}}

    \caption{The graph $G_4$}  \label{fig:TARdisconnect}
\end{figure}

\begin{thm}\label{t:TARdisconnect}
For $n\ge 3$, the graph $G_n$ has $\pdbar(G_n)=\pd(G_n)=n-1$ and $\ulpdo(G_n)=\pdo(G_n)=2n-2>\pdbar(G_n)+1=\pd(G_n)+1$.
\end{thm}
\bpf 
We begin by establishing that each of the following is a power dominating set: $T_n$, and any set consisting of exactly one vertex from all but one of the sets $S_{n,i}$. For $T_n$, note that  $N_{G_n}[T_n]=V_n\setminus \{v^i_n: 1\leq i\leq n\}$ by the definition of $G_n$. Thus $\{v^i_n: 1\leq i\leq n\}$ are observed in the second round, and it follows that $T_n$ is a power dominating set of $G_n$. Let  $A\subset \cup_{i=1}^n S_{n,i}$  such that $|A|=n-1$ and  $|A\cap S_{n,i}|\leq 1$ for any $i=1,\dots ,n$.
The conditions imposed on $A$ imply it contains one and only one vertex in $S_{n,i}$, for all $i \in \{1,\ldots, n\}$ with the exception of one particular integer, say $p$. This means that  $V_n\setminus N_G[A]\subseteq T_n \cup S_{n,p}$.  Since $n\geq 3$, there exists $i$ such that all vertices in $S_{n,i}$ are observed in round $1$, and since each vertex in $S_{n,i}$ has at most one neighbor in $T_n$, any vertex in $T_n$ not observed in round $1$ is observed in round $2$. Then, after two propagation rounds we can assume the only unobserved vertices of $G_n$ are those in $S_{n,p}$.  Since all vertices in $T_n$ are observed after round 2, and each of them has one and only one neighbor in $S_{n,p}$,  $v^p_j$ is observed after round  $3$ for $j=1,\dots,n-1$. As a result, after round $3$ there is exactly one non-observed vertex, $v^p_n$, which is observed in the following propagation round. Hence, $A$ is a power dominating set. 

Let $U\subseteq V(G_n)$ such that $U\cap S_{n,i}=\emptyset$, $U\cap S_{n,j}=\emptyset$, $i\ne j$, and $T_n\not\subseteq U$.  We show that $U$ is not a power dominating set: Since $T_n\not\subseteq U$, there is a $k$ such that $u_k\not\in U$. So after the first round, $v_k^i, v_n^i,v_k^j$ and $v_n^j$ are all unobserved.  Every observed vertex in $S_{n,i}$ is adjacent to both $v_k^i$ and $v_n^i$, and similarly for $j$.  Furthermore, $u_k$ is adjacent to both  $v_k^i$ and $v_k^j$, 
and every other vertex of $T_n$ is adjacent to neither  $v_k^i$ nor $v_k^j$. 
Thus $v_k^i, v_n^i,v_k^j$ and $v_n^j$ are never observed. 
This implies a set $R\subseteq V$ is a power dominating set if and only if $R$ contains $T_n$ or at least one vertex from all but one of the sets  $S_{n,i}$.
Therefore, $\pd(G_n)=\pdbar(G_n)=n-1$.

By Corollary \ref{basic1-pd}, $\ulpdo(G)\le \pdo(G)\le \pdbar(G)+\pd(G)=2n-2$.  To see that $\ulpdo(G)\ge 2n-2$, it is sufficient to observe that any path between $T_n$ and $A=\{v_1^1,v_1^2,\dots,v_1^{n-1}\}$ must contain $T_n\cup A$.
\epf


\subsection{Uniqueness}
We now consider graphs which are characterized by their reconfiguration graph. Since for any graph $G$ adding isolated vertices does not change $\tar(G)$, we consider only graphs with no isolated vertices. For a graph $G$ with no isolated vertices, we say $\tar(G)$ is {\it unique} if for any graph $H$ with no isolated vertices,  $\tar(G)\cong\tar(H)$ implies $H\cong G$. 

\begin{prop}\label{p:upperpdconnected}
Let $G$ be a connected graph on $n\geq4$ vertices. Then $\pdbar(G)\leq n-2$. Furthermore, $\pdbar(G)=n-2$ if and only if $G=K_{1,n-1}$.
\end{prop}

\bpf
Since $G$ has no isolated vertices, $\pdbar(G)\leq n-1$ by Definition \ref{d:Xprop}. 
If $\pdbar(G)=n-1$, choose a minimal  power dominating set $S$ of order $n-1$. Let $w$ be the vertex not in $S$. Assume $w$ is observed by vertex $v\in S$. If $w$ or $v$ had an additional neighbor $y$, then $S-y$ would power dominate, contradicting the minimality of $S$. The edge $vw$ is a connected component of $G$, hence $G$ is not connected. Therefore, $\pdbar(G)\leq n-2$ if $G$ is connected.

Let $S$ be a power dominating set with $|S|=n-2$ and let $v$ and $w$ be the vertices not in $S$. First suppose that $v$ and $w$ are not adjacent. If they share a neighbor $x \in N(v)\cap N(w)$, then there must exist another vertex $y \in V(G)$ connected to at least one of $v$, $w$ or $x$, so $S\setminus\{y\}$ is a power dominating set and $S$ is not minimal.  The other possibility is the $v$ and $w$ have no common neighbors. Suppose $x \in N(v)$ and $y \in N(w)$. If $n=4$, $G$ must be the path $P_4$ with leaves $v$ and $w$, but then $S=\{x,y\}$ is not minimal. In the case $n\geq5$, at least one of $\{v,w,x,y\}$ has an extra neighbor $u$, and again $S\setminus\{u\}$  is a power dominating set and $S$ is not minimal.  We have shown that if $S$ is minimal then  the vertices $v$ and $w$ not in $S$ are adjacent.

Now consider the case in which  $v$ and $w$ are adjacent and $z \in N(v) \cap N(w)$. There is at least one vertex $y \in S$ adjacent to $v$, $w$ or $z$. Then, $z$ dominates $v$ and $w$, and any neighbor of $y$ will observe $y$, so $S\setminus\{y\}$ is a  power dominating set and  $S$ is not not minimal. 

Now assume $S$ is minimal, so $v$ and $w$ are adjacent and $ N(v) \cap N(w)=\emptyset$.  Assume without loss of generality that $v$ has a neighbor $x$. If $w$ had a neighbor $y$, then $S\setminus\{y\}$ is a  power dominating set and  $S$ would not be not minimal.  So  $w$ is a leaf in $G$.   If $x$ had another neighbor besides $v$, it would be unnecessary to include it in $S$, so $x$ is also a leaf in $G$. Therefore, $v$ has at least one extra neighbor, $y$. Obviously, it is necessary for $N(v)\setminus \{w\}$ to be observed for $v$ to observe $w$. However, if $y$ were adjacent to $z$, then $z$ could observe $y$ (whether $z$ is adjacent to $v$ or not),  and thus $S\setminus\{y\}$ would be a power dominating set, contradicting the minimality of $S$. Therefore, the neighbors of $v$ must be leaves in $G$, and hence $G=K_{1,n-1}$. 
\epf

Applying Lemma \ref{L:uembedQt} and Corollary \ref{c:X-Qt} to power domination   yields the next two results.


\begin{cor}\label{L:embedQt}
Let $G$ be a graph  on $n\geq3$ vertices and let $2\leq t \leq n-1$. Then $\tar(G)$ has an induced subgraph isomorphic to the hypercube $Q_t$ if and only if $\pd(G)\leq n-t$. 
\end{cor}


\begin{cor}\label{c:pdequal1}
For a graph $G$ on $n\geq3$ vertices, $\pd(G)=1$ if and only if $\tar(G)$ has an induced subgraph isomorphic to $Q_{n-1}$. 
\end{cor}

\begin{thm}\label{upperPDn-3} Let $G$ be a graph with  no isolated vertices, $r\ge 3$, and  $\tar(K_{1,r})\cong\tar(G)$. Then $G\cong K_{1,r}$.
\end{thm}

\bpf 
Since $\tar(G)$ has an induced subgraph isomorphic to $Q_{n-1}$,  $\pd(G)=1$ by Corollary \ref{c:pdequal1}.  Thus  $G$ is connected.  By Corollary \ref{c:pd-iso-props}, $\pdbar(G)=\pdbar(\tar(K_{1,n-1}))=n-2$. Then Proposition \ref{p:upperpdconnected} implies $G\cong K_{1,n-1}$.
\epf




We now characterize graphs with $\pdbar{G}=n-3$ and establish which of those have unique TAR reconfiguration graphs. Let $K_{1,t}(e)$ denote the graph constructed from $K_{1,t}$ by adding an edge between one pair of leaves,  $K_{1,t}(\ell)$ denote the graph constructed from $K_{1,t}$ by adding an additional vertex adjacent to one leaf, and  $K_{2,t}(e)$ denote the graph constructed from $K_{2,t}$ by adding an edge between  the vertices in the partite set of order $2$.
\begin{lem}\label{k2tminimalsets}
For $t\ge 3$,  denote the partite set of order $2$ in $K_{2,t}$ by $X$ and the other by $Y$.  Use the same notation for  $K_{2,t}(e)$. The minimal power dominating sets of each of the graphs $K_{2,t}$ and $K_{2,t}(e)$ have one of two forms: any one vertex in $X$ or any $t-1$ vertices in $Y$. 
Thus $K_{2,t}$ and $K_{2,t}(e)$ each have the same  minimal power dominating sets (there are $t+2$ of them). 
\end{lem}
\begin{proof}
It is straightforward to verify that any one vertex in $X$ is a minimal power dominating set, as is a set of $t-1$ vertices from $Y$.  By Proposition \ref{p:upperpdconnected}, all minimal power dominating sets have cardinality of $t-1$ or less.  Since each vertex in $X$ is a power dominating set, the only other possible minimal power dominating sets, then, are subsets of $Y$, but a subset of $Y$ is a power dominating set only if it contains at least $t-1$ vertices.  The last statement is immediate.
\end{proof}

\begin{lem}\label{k1teminimalsets}
 For $t\ge 4$,  let $v$ be the vertex of degree $t$ in $K_{1,t}(e)$ and let $u$ and $w$  be the two vertices  of degree $2$.  The minimal power dominating sets of $K_{1,t}(e)$  have one of two forms: 
 \begin{enumerate}
\item the set $\{v\}$, or 
\item a set of $t-2$ vertices containing one of $u$ or $w$ and all but one of the $t-2$ vertices of degree 1 adjacent to $v$.  
\end{enumerate}
Thus $K_{1,t}(e)$ has $2t-3$ minimal power dominating sets. 
\end{lem}

\begin{proof}
It is straightforward to verify that the set $\{v\}$, or  a set of $t-2$ vertices containing one of $u$ or $w$ and all but one of the $t-2$ vertices  adjacent to $v$ that have degree $1$ are minimal power dominating sets.  By Proposition \ref{p:upperpdconnected}, all minimal power dominating sets have cardinality at most $t-2$.   Since $\{v\}$ is a power dominating set, the only other possible minimal power dominating sets are subsets of $V(K_{1,t}(e)) \backslash \{v\}$.  Any power dominating set not containing $v$ must contain at least one of $u$ or $w$, and all but one of the remaining vertices. The last statement follows from $1+2(t-2)=2t-3$. 
\end{proof}

Note that the order of $K_{1,t}(e)$ is $t+1$ whereas  $K_{1,t}(\ell)$, $K_{2,t}$, and  $K_{2,t}(e)$ have order $t+2$.  For comparison, we observe that $K_{1,t+1}(e)$ has order $t+2$ and $2t-1$ minimal power dominating sets.

\begin{lem}\label{k1tellminimalsets}
For $t\ge 3$,  let $v$ be the vertex of degree  $t$ in $K_{1,t}(\ell)$, let $u$ be the only vertex of degree $2$, and $w$ be its leaf neighbor. The minimal power dominating sets of   $K_{1,t}(\ell)$ have one of two forms:
\begin{enumerate}
\item the set $\{v\}$, or 
\item a set of $t-1$ vertices not containing $v$ and  containing at most one of $u$ or $w$.
\end{enumerate}
Thus $K_{1,t}(\ell)$ has $2t$ minimal power dominating sets.\end{lem}

\begin{proof}
It is straightforward to verify that the set $\{v\}$ and a set of $t-1$ vertices not containing $v$ and  containing at most one of $u$ or $w$ are minimal power dominating sets.  By Proposition \ref{p:upperpdconnected}, all minimal power dominating sets have cardinality  $t-2$ or less.   Since $\{v\}$ is a power dominating set, the only other possible minimal power dominating sets are subsets of $V(K_{1,t}(e)) \backslash \{v\}$, but note that any such set must have order at least $t-1$ and must have at least $t-2$ of the $t-1$ degree 1 neighbors of $v$ to be a power dominating set.  The last statement follows from $1+2(t-1)+1=2t$. 
\end{proof}

\begin{thm}\label{t:pdbarnless3}
For any connected graph $G$ with  {$|V(G)|\ge 6$}, $\pdbar(G)=n-3$ if and only if $G$ is one of the following graphs: 
 $K_{1,t}(e)$,  $K_{1,t}(\ell)$, $K_{2,t}$, or  $K_{2,t}(e)$.
\end{thm}
\begin{proof}
Lemma \ref{k2tminimalsets} through \ref{k1tellminimalsets} establish that $\pdbar(G)=n-3$ for each graph $G$ on the list.  

We now show that these are the only connected graphs on six or more vertices with $\pdbar(G)=n-3$.  Let $S$ be a minimal power dominating set with $|S|=n-3$.  Let $V\backslash S = \{u, v, w\}$.

\underline{Case 1:} First, suppose that $N(\{v, w\}) \cap S = \emptyset$.  That is, only $u$ is adjacent to any vertex in $S$.  If $v$ and $w$ were both adjacent to $u$, then $S$ would not be a power dominating set, so  the graph induced by $V \backslash S$ is a path with $u$ as an endpoint.  By Proposition \ref{lem:neighbors},  $u$ is adjacent to every vertex in $S$.  Suppose $xy \in E(G)$ for $x, y \in S$.  Then $S \backslash \{x\}$ is a power dominating set of $G$, contradicting the minimality of $S$.  Hence, there are no edges in $S$, so $G=K_{1,t}(\ell)$. 

\underline{Case 2:}  
Next, suppose that $N(w) \cap S = \emptyset$, but both $u$ and $v$ have neighbors in $S$.  Without loss of generality, $uw \in E(G)$.  Let $S_u$ be the set of vertices in $S$ that are adjacent to $u$ but not to $v$, let $S_v$ be the set of vertices in $S$ that are adjacent to $v$ but not to $u$, and let $S_{uv}$ be the set of vertices in $S$ that are adjacent to both $u$ and $v$.

First assume that $vw \notin E(G)$.  Note that  $|S_v| \leq 1$, since otherwise we could remove one vertex from  $S_v$ and still produce a power dominating set, violating minimality.  {Also, if $|S_{uv}| >1$, then removing one vertex from $S_{uv}$ results in a power dominating set, giving us that $|S_{uv}|=1$. Furthermore, there are no edges with both endpoints in $S_u\cup S_{uv}$,} since otherwise we could remove one vertex from {$S_u$}  and still produce a power dominating set, violating minimality.  If $S_{uv} = \emptyset$, then $uv\in E(G)$, $|S_v|=1$, and $G=K_{1,t}(\ell)$.  
Thus we assume that $S_{uv} \neq \emptyset$, in which case $S_v = \emptyset$, since removing {the vertex in $S_v$} from $S$ would still result in a power dominating set, violating minimality.    If $uv \in E(G)$, then $G=K_{1,t}(e)$, and if not, $G=K_{1,t}(\ell)$. 

Continuing with the case that $N(w) \cap S = \emptyset$ and both $u$ and $v$ have neighbors in $S$, we now assume that $uw, vw \in E(G)$.   First assume $S_{uv} = \emptyset$, which implies  $S_u$ and $S_v$ are both nonempty.  If $|S_u|>1$, then omitting one vertex from $S$ in  $S_u$  results in a power dominating set, contradicting minimality, and  $|S_v|>1$ is similar.  Thus,  $|S_u|=|S_v|=1$, which implies $|V(G)|=5$, violating our assumption that $|V(G)| \geq 6$. 

To conclude this case, we assume  $N(w) \cap S = \emptyset$, both $u$ and $v$ have neighbors in $S$, that $uw, vw \in E(G)$, and $S_{uv} \neq \emptyset$.  Then $S_u, S_v$ are both  empty; otherwise, omitting one vertex from either results in a power dominating {set,} violating minimality of $S$.  Also, there are no edges between any pair of vertices in $S_{uv}$ since we could omit one end vertex of the edge from $S$ and produce a power dominating set, again violating minimality of $S$.  Thus,   $G=K_{2,t}$ or $G=K_{2,t}(e)$. 

\underline{Case 3:} Finally, suppose that each of $u, v, w$ has a neighbor in $S$.  Suppose first that  $x \in S$ is a neighbor of $u, v$, and $w$.  For any other vertex $y \in S$, $S \backslash \{y\}$ is a power dominating set, violating minimality of $S$. But if no such $y$ exists, then $|V(G)|=4$, contradicting our assumption that $|V(G)| \geq 6$.  
Thus, no one vertex in $S$ is adjacent to all of $u, v, w$.  {Let $X=\{x,y\}$ be a set of two  vertices in $S$ such that $N(X) \backslash S = \{u, v, w\}$. Without loss of generality, $x$ is adjacent to two of $\{u,v,w\}$, but then  $S \backslash \{y\}$ is a power dominating set, violating minimality.  
  The} only case remaining is that $X = \{x, y, z\}$, with $u$ the only neighbor of $x$, $v$ the only neighbor of $y$, and $w$ the only neighbor of $z$.  (Note an edge among $x, y, z$ would violate minimality).  In this case as well, $\{u, v,w\}$ must induce a path or cycle in $G$ in order for $G$ to be connected; in either case, $\{x, y\}$ is a power dominating set, violating the minimality of $S$.  

Hence, the only connected graphs on at least six vertices that have $\pdbar(G)=n-3$ are $K_{1,t}(e)$,  $K_{1,t}(\ell)$, $K_{2,t}$, and  $K_{2,t}(e)$. 
\end{proof}

Since the TAR graph of a graph $G$ is completely determined by the minimal power dominating sets of $G$, the next result is immediate from Lemma \ref{k2tminimalsets} for $t\ge 3$.  Furthermore, any one vertex is a power dominating set for $ \tar(K_{2,2})$ and $\tar(K_{2,2}(e))$, so $\tar(K_{2,2}(e)) \cong \tar(K_{2,2})$.

\begin{cor}
\label{mPDS:n-3:K2t}
If $t \geq 2$, then  $\tar(K_{2,t}(e)) \cong \tar(K_{2,t})$.  That is, the TAR reconfiguration graphs of $K_{2,t}$ and $K_{2,t}(e)$ are not unique.
\end{cor}

\begin{prop}
\label{mPDS:n-3:K1t} Let $t \geq 4$ and let $H$ be a graph with no isolated vertices. 
If   $\tar(K_{1,t+1}(e)) \cong \tar(H)$, then $H \cong K_{1,t+1}(e)$.  If  $\tar(K_{1,t}(\ell)) \cong \tar(H)$, then $H \cong K_{1,t}(\ell)$. 
\end{prop}

\begin{proof}
Suppose $\tar(K_{1,t+1}(e)) \cong \tar(H)$.  By Corollary \ref{c:pd-iso-props}, $\pd(H)=\pd(K_{1,t}(e))=1$.  This implies  that $H$ is connected.  
By Theorem \ref{t:pdbarnless3}, the only possibilities are that $H$ is isomorphic to one of  $K_{1,t+1}(e)$, $K_{1,t}(\ell)$, $K_{2,t}$, or $K_{2,t}(e)$.  By examining the number of minimal power dominating sets in Lemmas \ref{k2tminimalsets}, \ref{k1teminimalsets}, and  \ref{k1tellminimalsets}, we see that $H \cong K_{1,t+1}(e)$.  The same reasoning holds for the case $\tar(K_{1,t}(\ell)) \cong \tar(H)$. 
\end{proof}


\begin{prop}\label{p:K33} $\tar(K_{3,3})$ is   unique. 
\end{prop}

\begin{proof}
Let $G$ be a graph with no isolated vertices such that  $\tar(G)\cong \tar(K_{3,3})$. Then by Corollary \ref{c:pd-iso-props},  $|V(G)|=6$,  $\pdbar(G)=2$, and $\pd(G)=2$. The order of $\tar(G)$ is 57, so every subset of $G$ of cardinality at least 2 must be a power dominating set. This implies that $G$ is connected. 

To show $G$ is $3$-regular, observe that if there exists a vertex $v$ with $\deg(v)\geq 4$ then $\{v\}$ would be a power dominating set, contradicting $\pd(G)=2$. So $\Delta(G)\leq 3$, and since $G$ is a connected graph that is not a path or cycle,  $\Delta(G)=3$. Suppose $G$ has a vertex $v$ with $\deg_G(v)\le 2$ and let $u$ be a neighbor of $v$ in $G.$ Since $\{u,v\}$ is a power dominating set, it follows that $\{u\}$ is a power dominating set, contradicting that $\gamma_P(G)=2.$ Thus, $G$ is 3-regular. 

To see that $G$ is bipartite, suppose to the contrary that $G$ has an odd cycle. Note that any 3-regular graph on 6 vertices containing a 5-cycle must also contain a 3-cycle, so suppose that $G$ has a 3-cycle, $C=(a,b,c)$. Let the remaining vertices of $G$ be $d$, $e$, and $f$.  Since $G$ is 3-regular, we assume without loss of generality that $a$ is adjacent to  $d$, and that $e$ is adjacent to  $b$. Then $\{a\}$ is a power dominating set of $G$, contradicting $\gamma_P(G)=2.$  Therefore, $G$ has no odd cycles and is bipartite. Thus $G$ must be $K_{3,3}.$
\end{proof}

 By Theorem \ref{upperPDn-3}, $\tar(K_{1,t})$ is unique for $t\ge 3$, and by Proposition \ref{p:K33}, $\tar(K_{3,3})$ is unique. Sage code computations have confirmed $\tar(K_{s,t})$ is unique for $s\in\{3,4\}$ and $t\in\{3,4,5\}$. This suggests the next Conjecture.
\begin{conj} $\tar(K_{s,t})$ is   unique if and only if $s\ne 2$ and $t\ne 2$.
\end{conj}

\section{A universal approach to the TJ reconfiguration graph}\label{s:universal-TE}

As in Section \ref{s:universal}, we work with a graph parameter $X(G)$ defined as the minimum cardinality among $X$-sets. 
In this section we assume that $X$ is a property such that an isolated vertex of $G$ must be in every $X$-set of $G$ and establish results for the token jumping reconfiguration graph of $X$.

\begin{defn} The {\em token jumping (TJ) reconfiguration graph for $X$}, denoted by $\xte(G)$, has as its vertex set the set of all minimum $X$-sets of $G$, with an edge between vertices $S$ and $S'$ if and only if $S$ can be obtained from $S'$ by exchanging a single vertex.
\end{defn}

 Examples of suitable properties include $X(G)$ is the domination number $\gamma(G)$,
   the power domination number $\pd(G)$ (see Section \ref{s:TE}),  the zero forcing number $\Z(G)$ (see \cite{GHH}), and   the positive semidefinite zero forcing number $\Zp(G)$.

\begin{rem}\label{no-isol-vtx-te} Suppose $G'=G\sqcup rK_1$ and $G$ has no isolated vertices.   In view of the assumption about $X$ in this section,   $X(G')=X(G)+r$,  and $\xte(G')\cong\xte(G)$.   Thus we  often assume a graph $G$ has no isolated vertices. 
\end{rem}
The proof of the next result is analogous to the proof of Proposition 2.7 in \cite{GHH}.

\begin{prop}\label{disjoint_union-x}
For any graphs $G$ and $G'$, $\xte(G \sqcup G') = \xte(G) \cp \xte(G')$.
\end{prop}

\begin{proof} Every vertex in $\xte(G \sqcup G')$ is of the form $S \sqcup S'$, where $S$ is a vertex of $\xte(G)$ and $S$ is a vertex of $\xte(S')$. Two vertices $S_1 \cup S'_1$ and $S_2 \cup S'_2$ are adjacent in $\xte(G \sqcup G')$ if and only if either $S_1$ and $S_2$ are adjacent in $\xte(S)$ and $S'_1 = S'_2$ or $S'_1$ and $S'_2$ are adjacent in $\xte(G')$ and $S_1 = S_2$. Thus $\xte(G \sqcup G') = \xte(G) \cp \xte(G')$. \end{proof}

\begin{cor}\label{disjoint_connected-x}
If $G$ and $G'$ are graphs such that $\xte(G)$ and $\xte(G')$ are connected, then $\xte(G \sqcup G')$ is connected.
\end{cor}

\begin{cor}\label{c:hypercube} If $\xte(K_2)=K_2$, then $\xte(K_2\sqcup\dots\sqcup K_2)=Q_d$.
\end{cor}

Observe that for $X(G) \in \{\Z(G), \Zp(G), \pd(G), \gamma(G)\}$,
$\xte(K_2)=K_2$, so $\xte(K_2\sqcup\dots\sqcup K_2)=Q_d$.

The proof of the next result is analogous to the proof of Proposition 2.4 in \cite{GHH}.

\begin{prop}\label{p:Delta}
Let $G$ be a graph such that $\xte(G)$ does not have a $K_3$ subgraph. Then \[\Delta(\xte(G))\le X(G).\]  
 \end{prop}
\bpf Let  $S\in V(\xte(G))$ such that $\deg_{\xte(G)}(S)=\Delta(\xte(G))$.  In order to be a neighbor of $S$, a minimum $X$-set of $G$ must differ by exactly one vertex from $S$.    If $S'\subsetneq  S$ and there exist distinct vertices $a,b\not \in S$ such  that $S'\cup\{a\}$ and $S'\cup\{b\}$ are minimum $X$-sets, then \{$S$, $S'\cup\{a\}$,  $S'\cup\{b\}$\} induces a $K_3$ in $\xte(G)$.  Thus each subset of $S$  of order $X(G)-1$ appears in at most one minimum $X$-set other than $S$.  Since there are exactly $X(G)$ subsets of $S$ with $X(G)-1$ vertices, there are at most $X(G)$ minimum $X$-sets  adjacent to $S$.  Thus $\Delta(\xte(G))\le X(G)$.
\epf

The hypothesis ``$\xte(G)$ does not have a $K_3$ subgraph" is necessary in token jumping reconfiguration graphs for zero forcing (shown in \cite{GHH}) and for power domination (shown in the next section).

Unlike the TAR model, graphs of different orders can have the same token jumping reconfiguration graphs (see examples for zero forcing in \cite{GHH}  and  for power domination in the next section).

\section{Power domination token jumping (TJ) graphs}\label{s:TE}

In addition to applying universal results for the  token jumping reconfiguration graph for power domination, we introduce a technique specific to power domination for realizing particular graphs as token jumping reconfiguration graphs. 
\begin{obs}
If $G$ is a graph with $\pd(G)=1$, then $\te(G)=K_r$ where $r$ is the number of minimum power dominating sets.
\end{obs}

It is well known that if a vertex $v$ is adjacent to three or more leaves, then $v$ must be in every minimum power dominating set.

\begin{rem}\label{p:add-leaves}
Suppose $G$ is a graph, $v$ is a vertex of $G$, and $v$ is in every minimum power dominating set.  Let $\tilde G$  be a graph obtained from $G$ by adding $1$ or more leaves adjacent to $v$ so that at least $3$ leaves are adjacent to $v$.  Then $\te(\tilde G)=\te (G)$ because $G$ and $\tilde G$ have the same minimum power dominating sets (each of which must contain $v$).
\end{rem}

 The ability to add leaves and retain the same token jumping graph presents a challenge to finding power domination token jumping graphs that are unique.  Additional difficulty is presented by the fact that $\te(G)\cong\te(H)$ does not imply that $G$ and $H$ have the same order (see Example \ref{e:Kpq}).

\subsection{Power domination TJ reconfiguration graphs of specific families of graphs}

In this section we  determine the power domination token jumping reconfiguration graphs  of complete graphs, complete bipartite graphs, paths, cycles, and wheels.

\begin{ex}\label{e:Kn} $\te(G)=K_n$ for $G=K_n, P_n, C_n, W_n$, or any graph of order $n$ where any one vertex is a power dominating  set.
\end{ex}

\begin{ex}\label{e:Kpq} Complete bipartite graphs:  $\te(K_{1,n-1})=K_1$ for $n\ge 4$ because the center vertex (which has degree $n-1$) is the unique minimum power dominating set (and $\te(K_{1,n-1})=K_n$ for $n=1,2,3$). $\te(K_{2,n-2})=K_2$ for $n\ge 5$ because either of the two vertices in the partite set of two is a minimum power dominating set (and $\te(K_{2,2})=K_4$).  
To see that $\te(K_{a,b})=K_a\cp K_b$ for $3\le a\le b\le n-3$, denote the vertices of $K_{a,b}$ by $x_1,\dots,x_a$ and $y_1,\dots,y_b$.  Then every minimum power dominating set is of the form $\{x_i,y_j\}$ for $i=1,\dots,a$ and $ j=1,\dots,b$.  Furthermore, $\{x_k,y_\ell\}$ is adjacent to $\{x_k,y_j\}$ for  $ j=1,\dots,\ell-1,\ell+1,\dots,b$ and to $\{x_i,y_\ell\}$ for $ i=1,\dots,k-1,k+1,\dots,a$. 
\end{ex}

As with $\te(K_{1,n-1})=K_1$ for $n\ge 4$, if $G$ is constructed from a graph $G'$ by adding three or more leaves to each vertex of $G'$ (i.e., $G=G'\circ rK_1$ with $r\ge 3$), then $\te(G)=K_1$  because $V(G')$ is the unique minimum power dominating set.  


\subsection{Realizing specific graphs as power domination TJ reconfiguration graphs}


Recall that the hypercube is always realizable as a token jumping graph by Corollary \ref{c:hypercube}, and $K_n$ is realizable by many graphs (see Example \ref{e:Kn}).  
For $r\ge 2$ and $d\ge 1$, the {\em Hamming graph} $H(d,r)$ has as vertices all $d$-tuples of elements of $\{0,\dots,r-1\}$, with two vertices adjacent if and only if they differ in exactly one coordinate.  The Hamming graph 
$H(d,r)$ is isomorphic to $K_r\cp \cdots\cp K_r$ with $d$ copies of $K_r$; the hypercube $Q_d$ is $H(d,2)$.
\begin{prop}
Let $G'$ be a graph of order $d$ and $G=G'\circ 2K_1$.  Then $\te(G)=H(d,3)$.
\end{prop}
\bpf Denote the vertices of $G'$ by $\{v_1,\dots,v_d\}$ and denote the leaves adjacent to $v_i$ in $G$ 
by $x_i$ and $y_i$.  Then every minimum power dominating set of $G$ 
has the form $\{w_1,\dots,w_d\}$ where $w_i\in\{v_i,x_i,y_i\}$.  Define $f_i:\{0,1,2\} \to \{v_i,x_i,y_i\}$ by $f_i(0)=v_i$, $f_i(1)=x_i$, and $f_i(2)=y_i$. Then mapping the sequence $j_1\dots j_d$ to the set $\{f_1(j_1),\dots ,f_d(j_d)\}$ is an isomorphism from $H(d,3)$ to $\te(G)$.
\epf

For a   connected graph $G$, define the graph $K^{2,3}(G)$ to be the graph obtained from $G$  deleting each edge of $G$  and adding three vertices of degree 2 adjacent to the endpoints of the edge that was deleted (so each edge of $G$ is replaced by three subdivided edges in $K^{2,3}(G)$).  We retain the vertex labels of $G$ in $K^{2,3}(G)$, so $V(G)\subset V(K^{2,3}(G))$.  Figure \ref{f:K23K3circK1} shows $K^{2,3}(K_3\circ K_1)$. 
\begin{figure}[h!]
\begin{center}
\scalebox{.4}{\includegraphics[width=8in]{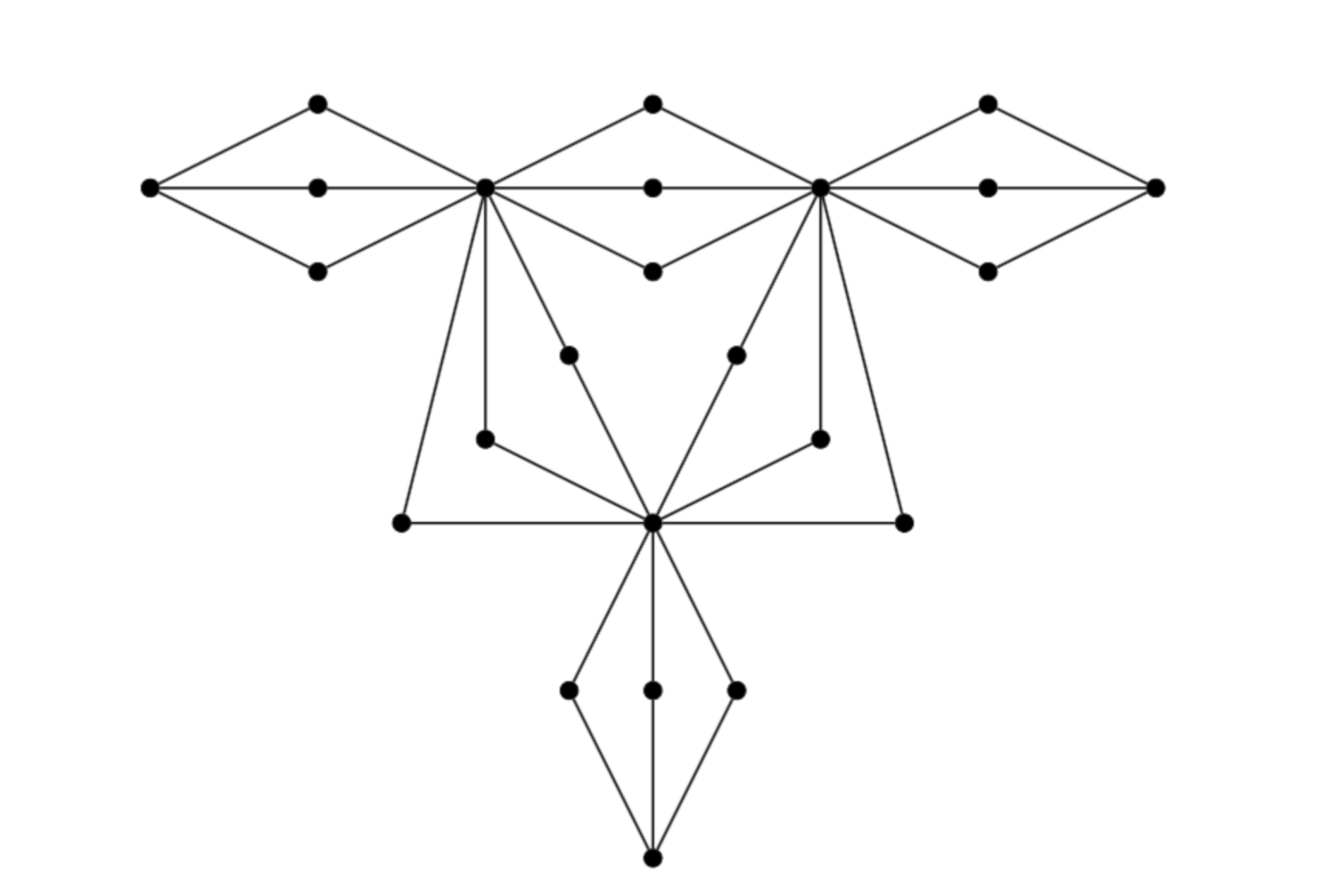}}
\caption{The graph $K^{2,3}(K_3\circ K_1)$.\label{f:K23K3circK1}}
\end{center}
\end{figure}

For a graph $G$,  a {\em vertex cover of $G$} (or a {\em vertex cover of the edges of $G$}) is a set of vertices $S$ such that every edge of $G$ is incident with a vertex in $S$.  

\begin{lem}\label{cover-K23G} Let $G$ be a graph.   If $S$ is a minimum power dominating set  of $K^{2,3}(G)$, then  $S\subseteq V(G)$. A set $S\subset V(G)$ is a  minimum vertex cover of  $G$ if and only if $S$ is a minimum power dominating set of $K^{2,3}(G)$.
 \end{lem}
\bpf
Denote the vertices of $G$ by $v_1,\dots,v_n$; note that $\{v_1,\dots,v_n\}\subset V(K^{2,3}(G))$.   Suppose first that $S$ is a power dominating set of $K^{2,3}(G)$. For every  $i,j\in\{1,\dots,n\}$ with $v_iv_j\in E(G)$,  one of $v_i$ or $v_j$ must be in $S$ or two of the three vertices in $N_{K^{2,3}(G)}(v_i) \cap N_{K^{2,3}(G)}(v_j)$ must be in $S$; if not, the other vertices  in $N_{K^{2,3}(G)}(v_i) \cap N_{K^{2,3}(G)}(v_j)$ can’t be observed.    
In order for $S$ to be a minimum power dominating set, $S$  must contain one of $v_i$ or $v_j$ for every $v_iv_j\in E(G)$ (and no vertices in $N_{K^{2,3}(G)}(v_i) \cap N_{K^{2,3}(G)}(v_j)$).  Thus $S$ is a vertex cover  of $G$. 

Now suppose $S$ is a minimum vertex cover of $G$ and let $v_iv_j\in E(G)$ with $v_i\in S$. Then the set of 3 vertices in $N_{K^{2,3}(G)}(v_i) \cap N_{K^{2,3}(G)}(v_j)$ is dominated and one of these vertices can observe $v_j$ if $v_j\not\in S$.  Thus $S$ is a power dominating set of $K^{2,3}(G)$. Since every minimum  power dominating set of $K^{2,3}(G)$ is a subset of $V(G)$, $S$ is a minimum  power dominating set of $K^{2,3}(G)$.
 \epf

\begin{prop} For $r\ge 3$, $\te(K^{2,3}(K_r\circ K_1))=K_{1,r}$.
 \end{prop}
\bpf
 Denote the vertices of $K_r$ by $v_1,\dots,v_r$ and let $u_i$ denote the leaf neighbor of $v_i$. Let $S$ be a minimum power dominating set of $K^{2,3}(K_r\circ K_1)$.  By Lemma \ref{cover-K23G}, either $v_i$ or $u_i$ must be in $S$ for $i\in\{1,\dots,r\}$, and   one of $v_i$ or $v_j$ must be in $S$ for every  $i,j\in\{1,\dots,r\}$ with $i\ne j$. Thus the minimum power dominating sets are $\{v_1,\dots,v_r\}$ and $\{v_1,\dots,v_{i-1},u_i,v_{i+1},\dots,v_r\}$, so $\te(K^{2,3}(K_r\circ K_1))=K_{1,r}$.  \epf

The next result shows that every path can be realized as a power domination token jumping graph (note $P_2=K_2$).
\begin{prop} For  $r\ge 2$, $\te(K^{2,3}(P_{2r}))=P_{r+1}$.
 \end{prop}
\bpf
 Denote the vertices of $P_{2r}$ by $v_1,\dots,v_{2r}$.   By Lemma \ref{cover-K23G}, the minimum power dominating sets of $K^{2,3}(P_{2r})$ are $\{v_1, v_3,\dots,v_{2r-1}\}$, $\{v_2, v_4,\dots,v_{2r}\}$, and $\{v_2, v_4,\dots,v_{2k},v_{2k+1},$ $v_{2k+3},\dots,v_{2r-1}\}$ for $k=1,\dots,r-1$.
\epf

\begin{prop}\label{TJoddcycle} For an odd integer $n \geq 3$, 
$\te(K^{2,3}(C_n))=C_n$.
\end{prop}

\bpf 
Let $n=2k+1$ for an integer $k \geq 1$, label the vertices of $C_n$ by the elements of $\mathbb{Z}_n$, and perform arithmetic modulo $n$. A minimum power dominating set for $K^{2,3}(C_n)$ (which is a minimum vertex cover of $C_n$) is a set of the form $S_a = \{ a \} \cup \{(a+(2j-1)):1 \leq j \leq k \}$ for all $a \in \mathbb{Z}_n$. Two sets $S_a$ and $S_b$ are adjacent in $\te(K^{2,3}(C_n))$ if and only if $b=a \pm 2$. The reconfiguration graph is connected because $\{ (a+2j): 0 \leq j <n \}=\mathbb{Z}_n$. Thus, $\te(K^{2,3}(C_n))=C_n$.
\epf

It is shown in \cite{GHH} that a star of order at least three cannot be realized as a zero forcing token jumping reconfiguration graph. 

\begin{quest} Is there a graph that cannot be realized as power domination token jumping reconfiguration graph?
\end{quest}


\subsection{Connectedness}
We now provide two examples of when the $\te(G)$ is disconnected. The next result follows from the proof of Theorem \ref{t:TARdisconnect}, which shows $T_n$ is an isolated vertex of $\te(G_n)$, where $G_n$ is defined in Definition \ref{e:bigupper}.

\begin{cor}
For $n\ge 3$,  $\te(G_n)$ is disconnected.
\end{cor}


Some cases of the grid graph also provide an example of $\te(G)$ being disconnected.

\begin{prop}
Let $G_{a,b}$ be the grid graph $G_{a,b}=P_{a}\cp P_{b}$ for $a=5,6,7$, or $8$ and $b\geq 12$. Then $\te(G_{a,b})$ is disconnected.
\end{prop}

\begin{proof} Consider the grid graph $G_{a,b}=P_{a}\cp P_{b}$ for $a=5,6,7$, or $8$.
and $b\ge 12$. It was shown in \cite{DH06} that $\pd(G_{a,b})=\lceil \frac{a}{4} \rceil=2$. Let the vertices of the graph be labeled $(x,y)$ for $1\leq x\leq a$ and $1\leq y\leq b$. Observe that $S_1=\{(2,1),(5,1)\}$ and $S_2=\{(2,b),(5,b)\}$ are power dominating sets for $G_{a,b}$ when $a=5$ or $a=6$ and $S_3=\{(2,3),(3,5)\}$ and $S_4=\{(2,b-2),(3,b-4)\}$ are power dominating sets for $G_{a,b}$ when $a=7$ or $a=8$. Note both vertices of $S_1$ and $S_3$ are contained in $V_1=\{(x,y) : 1\leq y\leq \lfloor \frac{b}{2}\rfloor\}$, and both vertices of $S_2$ and $S_4$ are contained in $V_2=\{(x,y) : \lfloor \frac{b}{2}\rfloor+1\leq y\leq b\}$, and $V_1$ and $V_2$ are disjoint. Assume for contradiction that $\te(G_{a,b})$ is connected. Then there must be a power dominating set $S=\{(x_1,y_1),(x_2,y_2)\}$ such that $(x_1,y_1)\subset V_1$ and $(x_2,y_2)\subset V_2$.

If $\dist((x_1,y_1),(x_2,y_2))>3$, then after the domination step, the set of observed vertices will induce a  disconnected graph. Thus, no further observation can occur since all observed vertices will have two or more unobserved neighbors (other than $(x_1,y_1)$ and $(x_2,y_2)$, which have no unobserved neighbors). Thus, $\dist((x_1,y_1),(x_2,y_2))\leq 3$ and so $(x_1,y_1)\subset \{(x,y) : \lfloor \frac{b}{2}\rfloor-2\leq y \leq \lfloor\frac{b}{2}\rfloor\}$ and $(x_2,y_2)\subset \{(x,y) : \lfloor \frac{b}{2}\rfloor+1\leq y \leq \lfloor\frac{b}{2}\rfloor+3\}$. We note both of these sets are contained within the subgraph with vertex set $V'=\{(x,y) : 1\leq x\leq a, \, \lfloor \frac{b}{2}\rfloor-6\leq y \leq \lfloor\frac{b}{2}\rfloor+6\}$ and this subgraph is isomorphic to $G_{a,12}=P_{a}\cp P_{12}$ for $a=5,6,7,8$. Since $S$ is a power dominating set for $G_{a,b}$ and is contained in $V'$, it is also a power dominating set for $G[V']$. Since $G[V']$ is isomorphic to $G_{a,12}$, there must exist an analogous power dominating set $S'=\{(x_1',y_1'),(x_2',y_2')\}$ for $G_{a,12}$ such that $(x_1',y_1')\subset \{(x,y) : 4\leq y \leq 6\}$ and $(x_2',y_2')\subset \{(x,y) : 7\leq y \leq 9\}$. 

Using {\it Sage} \cite{sage:pd-recon},
all power dominating sets of $G_{a,12}$ can be found for $a=5,6,7,8$. We observe that for all power dominating sets $S=\{(x_1,y_1),(x_2,y_2)\}$, either both $(x_1,y_1)$ and $(x_2,y_2)$ are contained in the set $V_1=\{(x,y) : 1\leq y\leq 6\}$ or $(x_1,y_1)$ and $(x_2,y_2)$ are both contained in the set $V_2=\{(x,y) : 7\leq y\leq 12\}$. So no power dominating set $S'$ of the desired form exists for $G_{a,12}$. Therefore, it must be that $\te(G_{a,b})$ is disconnected.
\end{proof}



\end{document}